\documentclass{amsart}
 \usepackage{hyperref}
\newtheorem{theorem}{Theorem}[section]
\newtheorem*{thmmain}{Theorem}
\newtheorem{lemma}[theorem]{Lemma}
\newtheorem{proposition}[theorem]{Proposition}

\newtheorem{corollary}[theorem]{Corollary}
\theoremstyle{definition}

\theoremstyle{remark}
\newtheorem{remark}[theorem]{Remark}

\numberwithin{equation}{section}

\begin{document}

\title[Compact convex ancient solutions of the planar affine normal flow]
 {Classification of compact convex ancient solutions of the planar affine normal flow}

\author[M.N. Ivaki]{Mohammad N. Ivaki}
\address{Institut f\"{u}r Diskrete Mathematik und Geometrie, Technische Universit\"{a}t Wien,
Wiedner Hauptstr. 8--10, 1040 Wien, Austria}
\curraddr{}
\email{mohammad.ivaki@tuwien.ac.at}
\date{}

\dedicatory{}

\begin{abstract}
We prove that the only compact convex ancient solutions of the planar affine normal flow are contracting ellipses.
\end{abstract}

\maketitle

\section{Introduction}
The setting of this paper is the two-dimensional Euclidean space, $\mathbb{R}^2.$ A compact convex subset of $\mathbb{R}^2$ with non-empty interior is called a \emph{convex body}. The set of smooth strictly convex bodies in $\mathbb{R}^2$ is denoted by $\mathcal{K}$ and write $\mathcal{K}_{0}$ for the set of smooth strictly convex bodies whose interiors contain the origin of the plane.

Let $K$ be a smooth strictly convex body in $\mathbb{R}^2$ and let $X_K:\partial K\to\mathbb{R}^2,$
be a smooth embedding of $\partial K$, the boundary of $K$. Write $\mathbb{S}^1$ for the unit circle and write $\nu:\partial K\to \mathbb{S}^1$ for the Gauss map of $\partial K$. That is, at each point $x\in\partial K$, $\nu(x)$ is the unit outwards normal at $x$.
The support function of $K\in\mathcal{K}_0 $ as a function on the unit circle is defined by
$s_K(z):= \langle X(\nu^{-1}(z)), z \rangle,$
for each $z\in\mathbb{S}^1$.
We denote the curvature of $\partial K$ by $\kappa$ which as a function on $\partial K$ is related to the support function by
 \[\frac{1}{\kappa(\nu^{-1}(z))}:=\mathfrak{r}(z)=\frac{\partial^2}{\partial \theta^2}s(z)+s(z).\]
Here and thereafter, we identify $z=(\cos \theta, \sin \theta)$ with $\theta$. The function $\mathfrak{r}$ is called the radius of curvature. The affine support function of $K$ defined by $\sigma:\partial K\to \mathbb{R}$ and $\sigma(x):=s(\nu(x))\mathfrak{r}^{1/3}(\nu(x))$ is invariant under the group of special linear transformations, $SL(2)$, and it plays a basic role in our argument.

Let $K\in \mathcal{K}$. A family of convex bodies $\{K_t\}_t\subset\mathcal{K}$ given by the embeddings $X:\partial K\times[0,T)\to \mathbb{R}^2$ is said to be a solution to the affine normal flow with the initial data $K_0$ if the following evolution equation is satisfied:
\begin{equation}\label{e: flow0}
 \partial_{t}X(x,t)=-\kappa^{\frac{1}{3}}(x,t)\, \nu(x,t),~~
 X(0,\cdot)=X_{K}.
\end{equation}
In this equation, $\nu(x,t)$ is the unit normal to the curve $ X(\partial K,t)=\partial K_t$ at $X(x,t).$
The well-known affine normal flow was investigated by Sapiro and Tannenbaum \cite{ST} and Andrews \cite{BA5,BA4}. Andrews proved that the affine normal flow evolves any convex initial bounded open set, after appropriate rescaling, exponentially fast, to an ellipsoid in the $\mathcal{C}^{\infty}$ topology. In another  direction, interesting results for the affine normal flow have been obtained in \cite{LT} by Loftin and Tsui regarding ancient solutions, and existence and regularity of solutions on non-compact strictly convex hypersurfaces. In \cite{LT}, the classification of compact, ancient solutions has been done in all dimensions except in dimension two: The only compact convex ancient solutions of the affine normal flow in $\mathbb{R}^n$, $n\geq3$, are contracting ellipsoids. We recall that a solution of the affine normal flow is called an ancient solution if it exists on $(-\infty, T)$,
for some finite time $T$. The main result of the paper is the following:
\begin{thmmain}
The only compact convex ancient solutions of the planar affine normal flow are contracting ellipses.
\end{thmmain}
To prove this theorem, we prove that the backwards limit of the affine support function is constant, hence showing the backwards limit of solution is an ellipse. On the other hand,
Andrews showed that the forwards limit is also an ellipse. Therefore, in view of the monotonicity of the area product $A(K_t^{\ast})A(K_t)$, we conclude that $A(K_t^{\ast})A(K_t)=\pi^2$ for all times.
This in turn implies that ancient solutions of the flow are contracting ellipses.

Similar classification has also been obtained by S. Chen \cite{Chen}. Notice that it was proved by Sapiro, Tannebaum and Andrews that any solution to the flow (\ref{e: flow0}) converges to a point in finite time. In the rest of this paper, we assume, without loss of generality, that the limiting point is the origin of the plane.
\section{Harnack estimate}
We first recall the following Harnack estimate from \cite{BA4}
 \begin{proposition}\label{prop: Harnack estimate}Along the flow (\ref{e: flow0}) we have
$\partial_{t}\left(\mathfrak{r}^{-\frac{1}{3}}t^{\frac{1}{4}}\right)\geq0$
on $\mathbb{S}^1\times (0,T)$.
\end{proposition}
\begin{corollary}\label{cor: harnack estimate all flow}
Every ancient solution of the flow (\ref{e: flow0}) satisfies
$\partial_{t}\left(\frac{\mathfrak{r}^{-\frac{1}{3}}}{s}\right)\geq 0.$
\end{corollary}
\begin{proof}
By Harnack estimate every solution of the flow (\ref{e: flow0}) satisfies
 \begin{equation*}
 \partial_{t}\mathfrak{r}^{-\frac{1}{3}}+\frac{1}{4t}\mathfrak{r}^{-\frac{1}{3}}\geq0.
 \end{equation*}
If we start the flow from a fixed time $t_0<0$, then this last inequality becomes
\begin{equation*}
 \partial_{t}\mathfrak{r}^{-\frac{1}{3}}+\frac{1}{4(t-t_0)}\mathfrak{r}^{-\frac{1}{3}}\geq0
\end{equation*}
for all $t_0<t<T.$
Letting $t_0$ goes to $-\infty$ yields $ \partial_{t}\mathfrak{r}^{-\frac{1}{3}}\geq0.$
Moreover, $s(\cdot,t)$ is decreasing on the time interval $(-\infty, T)$. The claim follows.
\end{proof}

\section{Affine differential setting}
We will now recall several definitions from affine differential geometry. Let $\gamma:\mathbb{S}^1\to\mathbb{R}^2$ be an embedded strictly convex curve with the curve parameter $\theta$. Define $\mathfrak{g}(\theta):=[\gamma_{\theta},\gamma_{\theta\theta}]^{1/3}$, where for two vectors $u, v$ in $\mathbb{R}^2$, $[u, v]$ denotes the determinant of
the matrix with rows $u$ and $v$. The affine arc-length is defined as
\begin{equation*}
\mathfrak{s}(\theta):=\int_{0}^{\theta}\mathfrak{g}(\alpha)d\alpha.
\end{equation*}
Furthermore, the affine normal vector $\mathfrak{n}$ is given by
$
\mathfrak{n}:=\gamma_{\mathfrak{s}\mathfrak{s}}.
$
In the affine coordinate ${\mathfrak{s}}$, there holds
$[\gamma_{\mathfrak{s}},\gamma_{\mathfrak{s}\mathfrak{s}}]=1,$
and $\sigma=[\gamma,\gamma_{\mathfrak{s}} ].$

Let $K$ be a smooth, strictly convex body having origin of the plane in its interior. We can express the area of $K$, denoted by $A(K)$, in terms of affine invariant quantities:
$$A(K)=\frac{1}{2}\int_{\partial K}\sigma d\mathfrak{s}.$$
The $p$-affine perimeter of $K\in \mathcal{K}_0$ (for $p=1$ the assumption $K\in \mathcal{K}_0$ is not necessary and we can take $K\in \mathcal{K}$), denoted by $\Omega_p(K)$, is defined by
\[\Omega_p(K):=\int_{\partial K}\sigma^{1-\frac{3p}{p+2}}d\mathfrak{s},\]
\cite{Lutwak2}. We call the quantity $\frac{\Omega_p^{2+p}(K)}{A^{2-p}(K)},$ the $p$-affine isoperimetric ratio and mention that it is invariant under $GL(2).$

Let $K\in\mathcal{K}_0$. The polar body associated with $K$, denoted by $K^{\ast}$, is a convex body in $\mathcal{K}_0$ defined by
\[
K^{\ast} = \{ y \in \mathbb{R}^{2} \mid x \cdot y \leq 1,\ \forall x
\in K \}.\]
The area of $K^{\ast}$, denoted by $A^{\ast}=A(K^{\ast})$ can also be
represented in terms of affine invariant quantities:
$$A^{\ast}=\frac{1}{2}\int_{\partial K}\frac{1}{\sigma^2}d\mathfrak{s}=
\frac{1}{2}\int_{\mathbb{S}^1}\frac{1}{s^2}d\theta.$$
Let $K\in \mathcal{K}_0$. We consider a family $\{K_t\}_t\subset \mathcal{K}$ given by the smooth embeddings $X:\partial K\times[0,T)\to \mathbb{R}^2$, which are
 evolving according to (\ref{e: flow0}). Then, up to a time-dependant diffeomorphism, $\{K_t\}_t$ evolves according to
\begin{equation}\label{e: affine def of flow}
 \frac{\partial}{\partial t}X:=\mathfrak{n},~~
 X(\cdot,0)=X_{K}(\cdot).
\end{equation}
Therefore, classification of compact convex ancient solutions of (\ref{e: flow0}) is equivalent to the classification of compact convex ancient solutions of
(\ref{e: affine def of flow}). In what follows, our reference flow will be always the evolution equation (\ref{e: affine def of flow}).
\begin{proposition}\label{cor: harnack estimate all flow1}
Every ancient solution to the flow (\ref{e: affine def of flow}) satisfies
$\left(\partial_{t}\sigma\right)\leq 0.$
\end{proposition}
\begin{proof}
Assume $Q$ and $\bar{Q}$ are two smooth functions $Q:\partial K\times [0,T)\to \mathbb{R}$, $\bar{Q}:\mathbb{S}^1\times[0,T)\to \mathbb{R}$ that are related by
$Q(x,t)=\bar{Q}(\nu(x,t),t)$.  To prove the claim, we observe that the identity $\frac{\partial}{\partial t} \nu=0$ (see \cite{LT} page 123) implies
$\partial_{t}\bar{Q}=\partial_{t}Q.$ Therefore, the claim follows from Proposition \ref{prop: Harnack estimate}.
\end{proof}
\begin{proposition}\label{prop: area backward limit}
Every ancient solution to the flow (\ref{e: affine def of flow}) satisfies
$$\lim\limits_{t\to-\infty}A(K_t)=\infty.$$
\end{proposition}
\begin{proof}
For a fixed $N>0$, in view of John's ellipsoid lemma, we may apply a special linear transformation $\Phi$ to $K_{-N}$, such that $B_R\subseteq\Phi K_{-N}\subseteq B_{2R}.$ Here, $B_{2R}$ and $B_R$ are disks of radii $R$ and $2R$ with the same centers. On the one hand, by comparing area, $B_R\subseteq\Phi K_{-N}\subseteq B_{2R}$  implies that $R\leq \sqrt{ \frac{A(\Phi K_{-N})}{\pi}}=\sqrt{ \frac{A(K_{-N})}{\pi}}.$ On the other hand, it is easy to see that it takes $B_{2R}$ exactly the time $\frac{3}{4}(2R)^{\frac{4}{3}}$ to shrink to its center under the affine normal flow. Since the solution $\{K_t\}$ exists on $[-N,T)$, by the comparison principle we must have
$$T+N\leq \frac{3}{4}(2R)^{\frac{4}{3}}\leq \frac{3}{4}\left(2\sqrt{ \frac{A(K_{-N})}{\pi}}\right)^{\frac{4}{3}}.$$
Therefore, $\lim\limits_{N\to+\infty}A(K_{-N})=\infty.$
\end{proof}
Next are two lemmas whose proofs were given in \cite{Ivaki}.
\begin{lemma}(Lemma 3.1 of \cite{Ivaki})\label{e: basic ev}
Let $\gamma_t:=\partial K_t$ be the boundary of the convex body $K_t$ which is evolving by the flow (\ref{e: affine def of flow}). Then the following evolution equations hold:
\begin{enumerate}
\item $\displaystyle \frac{\partial}{\partial t}\sigma=-\frac{4}{3}
+\frac{1}{3}\sigma_{\mathfrak{s}\mathfrak{s}},$
\item $\displaystyle \frac{d}{dt}A=-\Omega_1.$
\end{enumerate}
\end{lemma}
\begin{lemma}(Lemma 6.1 of \cite{Ivaki})\label{lem: controlling derivative of $l$-affine surface area along affine normal flow}
For every $l\geq2$  the $l$-affine perimeter satisfies:
\begin{equation}\label{e: eveq general}
\frac{d}{dt}\Omega_l(t)=\frac{2(l-2)}{l+2}\int_{\gamma_t}\sigma^{-\frac{3l}{l+2}}d\mathfrak{s}
+\frac{6l}{(l+2)^2}\int_{\gamma_t}\sigma^{-1-\frac{3l}{l+2}}(\sigma_{\mathfrak{s}})^2d\mathfrak{s}.
\end{equation}
\end{lemma}
\begin{remark}\label{rem: rem}
We recall from \cite{BA5} that $\frac{d}{dt}\frac{\Omega_1^3(K_t)}{A(K_t)}\geq 0$ and $\frac{d}{dt}A(K_t)A(K_t^{\ast})\geq0$. Moreover, equality holds in the latter case if and only if $K_t$ is an origin-centered ellipse.
\end{remark}
Write $\sigma_M(t)$ and $\sigma_m(t)$ for $\max\limits_{\gamma_t}\sigma$ and $\min\limits_{\gamma_t}\sigma$, respectively. In the next lemma, we will prove that the ratio
$\frac{\sigma_{M}(t)}{\sigma_m(t)}$ remains uniformly bounded from above.
\begin{lemma}
There is a constant $0<c<\infty$, such that $\frac{\sigma_{M}(t)}{\sigma_m(t)}\leq c$ for all $t\in(-\infty,0].$
\end{lemma}
\begin{proof}
Combining Proposition \ref{cor: harnack estimate all flow1} and part (1) of Lemma \ref{e: basic ev} yield
 \begin{align}\label{e: fundamental}
0&\geq\frac{\partial_t\sigma}{\sigma}=-\frac{4}{3\sigma}
+\frac{1}{3}\frac{\sigma_{\mathfrak{s}\mathfrak{s}}}{\sigma}.
 \end{align}
Integrating (\ref{e: fundamental}) against $d\mathfrak{s}$ gives
$
4\int_{\gamma_t}\frac{1}{\sigma}d\mathfrak{s}\geq
\int_{\gamma_t}
((\ln \sigma)_{\mathfrak{s}})^2d\mathfrak{s}.
$
By applying the H\"{o}lder inequality to the left-hand side and right-hand side of the inequality we obtain
\begin{align*}
\left(\int|\left(\ln\sigma\right)_{\mathfrak{s}}|d\mathfrak{s}\right)^2\leq\Omega_1\int_{\gamma_t}
((\ln \sigma)_{\mathfrak{s}})^2d\mathfrak{s}\le 4\Omega_1\int_{\gamma_t}\frac{1}{\sigma}d\mathfrak{s}&\leq\sqrt{32}A^{\ast~\frac{1}{2}}\Omega_1^{\frac{3}{2}}\\
&=\sqrt{32}\frac{A^{\frac{1}{2}}A^{\ast ~\frac{1}{2}}\Omega_1^{\frac{3}{2}}}{A^{\frac{1}{2}}}.
\end{align*}
Here, we used the identities $\int_{\gamma_t}\frac{1}{\sigma^2}d\mathfrak{s}=2A^{\ast}$
and $\int_{\gamma_t} d\mathfrak{s}=\Omega_1.$
Consequently, by Remark \ref{rem: rem} we get
$
\left(\ln\frac{\sigma_{M}(t)}{\sigma_m(t)}\right)^2\leq \sqrt{32}\frac{A^{\frac{1}{2}}(0)A^{\ast ~\frac{1}{2}}(0)\Omega_1^{\frac{3}{2}}(0)}{A^{\frac{1}{2}}(0)}.
$
This implies that
\begin{equation}\label{ie: sigma ratio}
\frac{\sigma_{M}}{\sigma_m}\leq c,
\end{equation}
on the time interval $(-\infty,0],$ for some constant $0<c<\infty$.
\end{proof}
Let $\{K_t\}_t$ be a solution of the flow (\ref{e: affine def of flow}). Then the family of convex bodies, $\{\tilde{K}_t\}_t$, defined by
$\tilde{K}_t:=\sqrt{\frac{\pi}{A(K_t)}}K_t$ is called a normalized solution to the affine normal flow, equivalently a solution that keeps the area of evolving bodies fixed and equal to $\pi.$

We furnish every quantity associated with the normalized solution with an over-tilde. For example, the support function, the curvature, and the affine support function of $\tilde{K}$ are denoted, in order, by $\tilde{s}$, $\tilde{\kappa}$, and $\tilde{\sigma}$.
\begin{lemma} There is a constant $0<c<\infty$ such that on the time interval $(-\infty,T)$
\begin{equation}\label{ie: uniform para1}
\frac{\tilde{\sigma}_{M}(t)}{\tilde{\sigma}_m(t)}\leq c.
\end{equation}
\end{lemma}
\begin{proof}
The estimate (\ref{ie: sigma ratio}) is scaling invariant. Therefore, the same estimate holds for the normalized solution.
\end{proof}
\begin{corollary}\label{cor: upper bound of invers ratio}
There exists a constant $0<b<\infty$ such that on $(-\infty, 0]$ we have $\frac{1}{\Omega_2^{4}(t)}<b.$
\end{corollary}
\begin{proof}
By the H\"{o}lder inequality:
\[\Omega_2(t)\Omega_1^{\frac{1}{2}}(t)(2A(t))^{\frac{1}{2}}\geq \Omega_2(t)
\int_{\gamma_t}\sigma^{\frac{1}{2}}d\mathfrak{s}=\int_{\gamma_t}\sigma^{-\frac{1}{2}}d\mathfrak{s}
\int_{\gamma_t}\sigma^{\frac{1}{2}}d\mathfrak{s}
\geq \Omega_1^2(t),\]
so
\[\Omega_2(t)\geq \left(\frac{\Omega_1^3(t)}{2A(t)}\right)^{\frac{1}{2}}=\left(\frac{\tilde{\Omega}_1^3(t)}{2\tilde{A}(t)}\right)^{\frac{1}{2}}.\]
Hence, to bound $\Omega_2(t)$ from below, it is enough to bound $\left(\frac{\Omega_1^3(t)}{2A(t)}\right)^{\frac{1}{2}}$ from below. Since
both $\Omega_2(t)$ and $\frac{\Omega_1^3(t)}{2A(t)}$ are invariant under $GL(2)$, we need only to prove the claim after applying appropriate $SL(2)$ transformations to the normalized solution of the flow.

By the estimate (\ref{ie: uniform para1}) and the facts that $\Omega_1(\tilde{K}_{t})=\pi^{\frac{1}{3}}\frac{\Omega_1(K_{t})}{A^{\frac{1}{3}}(K_t)}$ is non-decreasing (see Remark \ref{rem: rem}), and $A(\tilde{K}_t)=\pi$ we have
\[\frac{c}{2}\tilde{\sigma}_m(t)\tilde{\Omega}_1(0)\geq\frac{1}{2}\tilde{\sigma}_M(t)\tilde{\Omega}_1(0)\geq\frac{1}{2}\tilde{\sigma}_M(t)\tilde{\Omega}_1(t)\geq
\tilde{ A}(t)=\pi.\]
So we get $\tilde{s}\tilde{\mathfrak{r}}^{\frac{1}{3}}(t)\geq \frac{2\pi}{c\tilde{\Omega}_1(0)}$ on $(-\infty,0]$.

On the other hand, recall that under the area restriction $A(\tilde{K}_t)=\pi$, and in the light of John's ellipsoid lemma, we can apply a special linear transformation at each time such that the diameter of $\tilde{K}_t$, $D(\tilde{K}_t):=\max\limits_{z\in\mathbb{S}^1}(s_{\tilde{K}_t}(z)+s_{\tilde{K}_t}(-z)),$ is bounded from above by $4.$ Therefore, as $\tilde{s}(t)>0$ (notice that the origin of the plane belongs to int $\tilde{K}_t$), we have, after applying a special linear transformation at each time, that $\tilde{s}(t)<4.$
So, in view of $\tilde{s}\tilde{\mathfrak{r}}^{\frac{1}{3}}(t)\geq \frac{2\pi}{c\tilde{\Omega}_1(0)}$, and that the affine support function is  $SL(2)$-invariant under, we get
$$4\tilde{\mathfrak{r}}^{\frac{1}{3}}(t)\geq\tilde{s}\tilde{\mathfrak{r}}^{\frac{1}{3}}(t)\geq \frac{2\pi}{c\tilde{\Omega}_1(0)}\Rightarrow
\tilde{\mathfrak{r}}^{\frac{2}{3}}(t)\geq\left(\frac{\pi}{2c\tilde{\Omega}_1(0)}\right)^2,$$
and thus
\begin{align}\label{e:11}
\Omega_2^2(t)\geq\frac{\tilde{\Omega}_1^3(t)}{2\tilde{A}(t)}=\frac{(\int_{\mathbb{S}^1}\tilde{\mathfrak{r}}^{2/3}d\theta)^3}{2\pi}\geq
(2\pi)^2\left(\frac{\pi}{2c\tilde{\Omega}_1(0)}\right)^6.
\end{align}
\end{proof}
\begin{lemma}\label{cor: liminf idea} If $K_t$ evolves by (\ref{e: affine def of flow}), then the following limit holds:
\begin{equation}
 \liminf_{t\to -\infty}\left(\frac{A(t)}{\Omega_1(t)\Omega_2^{5}(t)}\right)
\int_{\gamma_t}\left(\sigma^{-\frac{1}{4}}\right)_{\mathfrak{s}}^2d\mathfrak{s}=0.
 \label{eq:sup}
 \end{equation}
\end{lemma}
\begin{proof}
By Lemma \ref{lem: controlling derivative of $l$-affine surface area along affine normal flow} for $l=2$, and by $\frac{d}{dt}A(t)=-\Omega_1(t)$, we find
$$\frac{d}{dt}\frac{1}{\Omega_2^{4}(t)}= \delta
\frac{d}{dt}\ln(A(t))\left[\left(\frac{A(t)}{\Omega_1(t)\Omega_2^{5}(t)}\right)
\int_{\gamma_t}\left(\sigma^{-\frac{1}{4}}\right)_{\mathfrak{s}}^2d\mathfrak{s}\right],$$
for a constant $\delta>0.$
Suppose on the contrary that there exists an $\varepsilon>0$ small enough such that
$$\left(\frac{A(t)}{\Omega_1(t)\Omega_2^{5}(t)}\right)
\int_{\gamma_t}\left(\sigma^{-\frac{1}{4}}\right)_{\mathfrak{s}}^2d\mathfrak{s}\geq \frac{\varepsilon}{\delta}$$
on $(-\infty, -N]$ for an $N>0$ large enough. Thus
$
\frac{d}{dt}\frac{1}{\Omega_2^{4}(t)}
\leq\varepsilon\frac{d}{dt}\ln(A(t)).
$
Integrating this last inequality on the time interval $[t,-N]$ and using Corollary \ref{cor: upper bound of invers ratio} we get
\begin{align*}
0<\frac{1}{\Omega_2^{4}(-N)}&\leq \frac{1}{\Omega_2^{4}(t)}+\varepsilon \ln(A(-N))-\varepsilon \ln(A(t))\\
&\leq b+\varepsilon \ln(A(-N))-\varepsilon \ln(A(t)).
\end{align*}
We reach to a contradiction by letting $t\to-\infty$: Since $\lim\limits_{t\to-\infty}A(t)=\infty$ (Proposition \ref{prop: area backward limit}), so the right-hand side becomes negative.
\end{proof}
\begin{corollary}\label{cor: limit of sigma}
There is a sequence of times $\{t_k\}_{k\in\mathbb{N}}$ and there is a constant $0<\zeta<\infty$ such that as $t_k\to-\infty$, then
$\lim\limits_{t_k\to-\infty}\max\limits_{\tilde{\gamma}_{t_k}}|\tilde{\sigma}(t_k)-\zeta|=0.$
\end{corollary}
\begin{proof}
Notice that the quantity $\left(\frac{A(t)}{\Omega_1(t)\Omega_2^{5}(t)}\right)\int_{\gamma_t}\left(\sigma^{-\frac{1}{4}}\right)_{\mathfrak{s}}^2d\mathfrak{s}$ is $GL(2)$ invariant and  also $\frac{\tilde{A}(t)}{\tilde{\Omega}_1(t)\tilde{\Omega}_2^{5}(t)}$ is bounded from below on $(-\infty,0]$: In fact, $\tilde{A}(t)=\pi$ and both quantities $\tilde{\Omega}_1(t)$
and $\tilde{\Omega}_2^{5}(t)$  are non-decreasing along the normalized flow thus $\tilde{\Omega}_1(t)\leq \tilde{\Omega}_1(0)$ and $\tilde{\Omega}_2^{5}(t)\le \tilde{\Omega}_2^{5}(0).$
Hence, Lemma \ref{cor: liminf idea} implies
that there exists a sequence of times $\{t_k\}$ such that $\lim\limits_{k\to\infty}t_k=-\infty$ while
$\lim\limits_{t_k\to-\infty}\int_{\tilde{\gamma}_{t_k}}
\left(\tilde{\sigma}^{-\frac{1}{4}}\right)^2_{\tilde{\mathfrak{s}}}d\tilde{\mathfrak{s}}=0.$
On the other hand, by the H\"{o}lder inequality
\begin{align*}
\frac{\left(\tilde{\sigma}_M^{-\frac{1}{4}}(t_k)-\tilde{\sigma}_m^{-\frac{1}{4}}(t_k)\right)^2}
{\tilde{\Omega}_1(t_k)}\leq \int_{\tilde{\gamma}_{t_k}}
\left(\tilde{\sigma}^{-\frac{1}{4}}\right)^2_{\tilde{\mathfrak{s}}}d\tilde{\mathfrak{s}}.
\end{align*}
So by boundedness of $\tilde{\Omega}_1$ from above we find that
$\lim\limits_{t_k\to-\infty}\left(\tilde{\sigma}_M^{-\frac{1}{4}}(t_k)-
\tilde{\sigma}_m^{-\frac{1}{4}}(t_k)\right)^2=0.$
To complete the proof, we need to exclude the possibility that $\limsup\limits_{t_k\to-\infty}\tilde{\sigma}_m(t_k)=\infty$
(Recall that in the proof of Corollary \ref{cor: upper bound of invers ratio},
we have already established the existence of a uniform lower bound for $\tilde{\sigma}_m$.). Suppose on the contrary $\limsup\limits_{t_k\to-\infty}\tilde{\sigma}_m(t_k)=\infty$.
In this case, (\ref{e:11}) gives
\[\pi=\tilde{A}(t_k)=\lim_{t_k\to-\infty}\tilde{A}(t_k)\geq \frac{1}{2}\left(\limsup\limits_{t_k\to-\infty}\tilde{\sigma}_m(t_k)\right)\tilde{\Omega}_1(t_k)=\infty.\]
This is a contradiction. So $\{\tilde{\sigma}_m\}$ is uniformly bounded from above and below. By passing to a further subsequence of $\{t_k\}$, if necessary, we obtain the desired result.
\end{proof}
\begin{lemma}\label{lem: selection}
For each time $t\in (-\infty, T)$, let $\Phi_t\in SL(2)$ be a special linear transformation such that the minimal ellipse containing $\Phi_t\tilde{K}_t$ is a disk. Then for an $N>0$ large enough, there holds
\[c_2 \frac{\zeta^{\frac{1}{3}}(2\zeta^{\frac{1}{3}})^{\frac{-3}{4}}}{2}\leq s_{\Phi_{t_k}\tilde{K}_{t_k}}\leq c_1 (2\zeta^{\frac{1}{3}})^{\frac{1}{4}}\]
 on the time interval $(-\infty,-N)$, for some absolute constants $0<c_1,c_2<\infty.$
\end{lemma}
\begin{proof}
The claim follows from Corollary \ref{cor: limit of sigma} and Corollary 2.4 of \cite{XW}.
\end{proof}
\section{Proof of the main Theorem}
\begin{proof}
By Lemma \ref{lem: selection} and by the Blaschke selection theorem, there is subsequence of times, denoted again by $\{t_k\}$, such that
$\{\Phi_{t_k}\tilde{K}_{t_k}\}$ converges in the Hausdorff distance to a convex body $\tilde{K}_{-\infty}$, having the origin in its interior, as $t_k\to-\infty.$ By Corollary \ref{cor: limit of sigma}, and by the weak convergence
of the Monge-Amp\`{e}re measures, the support function of $\tilde{K}_{-\infty}$ is the generalized solution of the following Monge-Amp\`{e}re equation on $\mathbb{S}^1$:
\[s^3(s_{\theta\theta}+s)=\zeta^3\]
Therefore, by Lemma 8.1 of Petty \cite{Petty}, $\tilde{K}_{-\infty}$ is an origin-centered ellipse. This in turn implies that $A(\tilde{K}_{-\infty})A(\tilde{K}_{-\infty}^{\ast})=\pi^2.$
On the other hand, $\lim_{t\to T}A(\tilde{K}_{t})A(\tilde{K}_{t}^{\ast})=\pi^2$, see Andrews \cite{BA5}. So the monotonicity of $A(\tilde{K}_{t})A
(\tilde{K}_{t}^{\ast})$ yields $A(\tilde{K}_{t})A(\tilde{K}_{t}^{\ast})\equiv \pi^2$ on $(-\infty,T)$ which gives $\frac{d}{dt}A(K_{t})A(K_{t}^{\ast})\equiv 0$ on $(-\infty,T).$
From Remark \ref{rem: rem} it follows that for every time $t$, $K_t$ and is an origin-centered ellipse.
\end{proof}
\textbf{Acknowledgment}
I would like to thank the referee for helpful comments and suggestions.
\bibliographystyle{amsplain}

\end{document}